\documentclass[12pt]{article}
\usepackage[dvips]{epsfig}
\usepackage{amscd}
\usepackage{amssymb}
\usepackage{amsthm}
\usepackage{amsmath}
\usepackage{latexsym}

\usepackage{hyperref}
\newtheorem{thm}{Theorem}[section]

\theoremstyle{definition}

\theoremstyle{remark}

\theoremstyle{plain}
\newtheorem{theorem} {Theorem}[section]

\newtheorem{lemma}[theorem]{Lemma}

\theoremstyle{definition}

\newtheorem{remark}[theorem]{Remark}

\newcommand{\norm}[1]{\ensuremath{\left\|#1\right\|}}
\newcommand{\abs}[1]{\ensuremath{\left|#1\right|}}

\newcommand{\Om}{\ensuremath{\Omega}}
\newcommand{\om}{\ensuremath{\omega}}
\newcommand{\pa} {\partial}
\newcommand{\pOm}{\ensuremath{\partial \Om}}

\newcommand{\eps}{\varepsilon}
\newcommand{\al} {\alpha}

\newcommand{\de} {\delta}

\newcommand{\De} {\Delta}
\newcommand{\la} {\lambda}

\newcommand{\f}{\frac}
\newcommand{\tf}{\tfrac}
\newcommand{\I}{\infty}

\newcommand{\R}{\ensuremath{\mathbb{R}}}

\newcommand{\be} {\begin{equation}}
\newcommand{\ee} {\end{equation}}
\newcommand{\bea} {\begin{eqnarray}}
\newcommand{\eea} {\end{eqnarray}}
\newcommand{\Bea} {\begin{eqnarray*}}
\newcommand{\Eea} {\end{eqnarray*}}

\numberwithin{equation}{section} \allowdisplaybreaks

\catcode`\@=11

\numberwithin{equation}{section} \allowdisplaybreaks
\begin{document}
%
%

\title{On a problem of resonance with exponential nonlinearity}

 \author{ B. B. Manna\footnote{ B. B. Manna, TIFR CAM , Bangalore, email: bhakti@math.tifrbng.res.in} ,
 P. N. Srikanth \footnote{P. N. Srikanth ,TIFR CAM , Bangalore, email: srikanth@math.tifrbng.res.in}}


\date{\today}
\maketitle
\begin{abstract} Consider the following semilinear elliptic problem on $B=\{x\in \R^2 : \abs{x}<1\}$
\begin{equation}
  \left\{\begin{aligned}
  -\De u &= \la_1u+e^u+f,  &&\mbox{ \qquad in } B \notag\\
   u &= 0  &&\mbox{   \qquad on } \partial B\\
  \end{aligned}
  \right.
 \end{equation}
with $f$ satisfying the following condition : $f$ is smooth integrable radial and satisfies \be 0<-\int_B f\phi_1<8\pi.\notag\ee Where $\phi_1$ is the eigen 
function of $(-\Delta)$ corresponding to the first eigenvalue $\la_1$ in $H_0^1(B)$. We shall find the existence of a radial solution of this
PDE. We shall use degree theory to get the existence starting from a suitable with known solution with its degree. Connecting those two PDE's 
by homotopy and getting the uniform estimate for the connecting PDE's we shall achieve our result. 
\end{abstract}

\section{Introduction}
Existence of solutions for semilinear elliptic Dirichlet problems
\begin{equation}
  \left\{\begin{aligned}
  -\De u &= g(x,u),  &&\mbox{ \qquad in } \Om \label{E0.0}\\
   u &= 0  &&\mbox{   \qquad on } \pOm\\
  \end{aligned}
  \right.
 \end{equation}
with distinct behavior of \be \f{g(x,s)}{s} \mbox{ as } s\to \pm\I\notag\ee is difficult to establish in the case when
\begin{itemize}
 \item [I.] $g(x,0)\neq 0$ (so there is no trivial solution).
 \item [II.] There is resonance in one direction, and
 \item[III.] The problem is superlinear in the other.
\end{itemize}
The problem seems to be particularly harder to deal if such a resonance is at the first eigenvalue of the laplacian, in view of the fact that 
the corresponding first eigenfunction has a definite sign. It is a problem of this kind that we will treat here, namely 
 \begin{equation}
  \left\{\begin{aligned}
  -\De u &= \la_1u+e^u+f,  &&\mbox{ \qquad in } B \label{E0.1}\\
   u &= 0  &&\mbox{   \qquad on } \partial B\\
  \end{aligned}
  \right.
  \end{equation}
 where $B=\{x\in \R^2 : \abs{x}<1\}$ with $f$ is a smooth integrable radial function and $\phi_1$ is the eigen function 
 of $(-\Delta)$ corresponding to the first eigenvalue $\la_1$ in $H_0^1(B)$. Solving \ref{E0.1} is particularly hard since $e^u$ is in some 
 sense the critical nonlinearity (see\cite{MR1132783}). Also see \cite{MR1989831} where the restriction on the exponent $p$ is essentially due 
 to lack of apriori bounds. In some sense, specially in the context of the approach we have adapted the result seems optimal (see \cite{MR2287880}),
 which seems to indicate that the bounds are hard to establish. See also the paper of McKenna and Rauch \cite{MR491053}, where apart from other 
 nonlinearities the case $-e^u$ has been studied. Note that in our case we have considered $+e^u$.
 
 Let us assume that $\phi>0$ in $B$ and 
 \be \int_B \phi_1^2=1 \text{, and } \int_B \abs{\nabla \phi_1}^2=\la_1\label{0.1}\ee Multiplying the equation $\ref{E0.1}$ by $\phi_1$ and
 integration by parts we get \be \int_B e^u\phi_1 + \int_B f\phi_1=0 \text{, i.e. } \int_B f\phi_1=-\int_B e^u\phi_1<0.\label{0.2}\ee The above 
 condition is necessary for the existence of the solution. The aim is to see if this is also sufficient condition. In this article we shall study
 the equation in the context of the radial solutions in $C^{1,\al}_{rad}(B)$ and shall prove the following result. 
 
 \begin{thm}\label{Th1.1}
  If \be 0<-\int_B f\phi_1<4\pi,\label{0.3}\ee then the equation(\ref{E0.1}) has a nontrivial radial solution in $C^{1,\al}_{rad}(B)$.
 \end{thm}
 Later in the last section we shall extend this result for the case $0<-\int_B f\phi_1<8\pi$ too.
 
 Our proof uses the well known degree and homotopy arguments. The required bounds for the homotopy, 
 established using the results of Brezis and Marle in \cite{MR1132783}. Let us recall the main result of Brezis, Marle, \cite{MR1132783}
 \begin{thm}{\emph{[BM-1]}}\label{ThBM.1}
  Assume $\; \Om \subset \mathbb{R}^{2} \;$ is bounded domain and let $\; u\;$ be
  a solution of  
   \begin{equation}
  \left\{\begin{aligned}
  -\Delta u &= f(x) &&\quad \mbox{in}\quad \Omega  \label{E:intro1}\\
   u &= 0 &&\quad \mbox{in}\quad \partial \Omega
   \end{aligned}
  \right.
  \end{equation}
   with $\; f \in L^{1}(\Omega) $. Set $\; \|f\|_{1} = \int_{\Omega} |f(x)|\; dx\; $ .\\ 
   For every $\; \delta \in (0, 4\pi)\;$ we have
  \begin{align} \label{E:intro2}
  \int_{\Omega} \exp \left[\f{(4\pi - \delta) |u(x)|}{\|f\|_{1}}\right]\, dx \leq
  \f{4 \pi^{2}}{\delta}
  (\text{diam}\; \Omega)^{2}
  \end{align}
 \end{thm}
 
 \begin{thm}{\emph{[BM-2]}}\label{ThBM.2}
  Let $\{u_n\}_n$ be a sequence of solutions of \be -\De u_n = V_n(x)e^{u_n}, \ \text{ in } \Om,\label{EBM}\ee where $\Om$ is a bounded domain
  in $\R^2$, satisfying $u_n=0$ on $\pa\Om$ such that \be \norm{V_n}_{L^p}\le C \text{ for some } 1<p\le\I,\label{BMC1}\ee and
  \be \int_\Om V_ne^{u_n}<\eps_0<4\pi/p' \text{ for all } n.\label{BMC2} \ee Then $\norm{u_n}_{L^\I}\le C$
 \end{thm}

 The next theorem is regarding the Uniform $L^\I$ bounds and blow-up, behavior for solutions of $-\De u = V(x)e^u$
 \begin{thm}{\emph{[BM-3]}}\label{ThBM.3}
  Let $\{u_n\}_n$ be a sequence of solutions of (\ref{EBM}), satisfying for some $1<p\le\I$
  \begin{itemize}
   \item [(i)] $V_n\ge 0$ in $\Om$,
   \item [(ii)] $\norm{V_n}_{L^p}\le C_1$,
   \item [(iii)] $\norm{e^{u_n}}_{L^{p'}}\le C_2$. 
  \end{itemize}
  Then there exists a subsequence $\{u_{n_k}\}$ satisfying one of the following alternatives
  \begin{itemize}
   \item [(a.)] $\{u_{n_k}\}_k$ is bounded in $L^\I_{loc}(\Om)$,
   \item [(b.)] $u_{n_k}\to -\I$ uniformly on compact subsets of $\Om$,
   \item [(c.)] the blow-up set $S$ (relative to $\{u_{n_k}\}$) is finite, nonempty and $u_{n_k}\to -\I$ uniformly on compact subsets of
                $\Om\backslash S$. In addition $V_{n_k}e^{u_{n_k}}$ converges in the sense of measures on $\Om$ to $\sum_i \al_i\de_{a_i}$
                with $\al_i\ge 4\pi/p'$ for all $i$ and $S=\cup_i\{a_i\}$.

  \end{itemize}
 \end{thm}

 \section{The comparison equation}
 Consider the equation 
 \begin{equation}
  \left\{\begin{aligned}
  -\De u = \la_1 u+g(u),  &\mbox{ \qquad in } B \label{E1.1}\\
   u = 0  &\mbox{   \qquad on } \partial B\\
  \end{aligned}
  \right.
 \end{equation}
 
 where 
 \begin{equation}
  g(t)=\left\{\begin{aligned}
  &sint &&\text{ if } -\pi\le t \le \pi \label{2.1}\\
  & 0 &&\text{ otherwise }  
 \end{aligned}
 \right.
\end{equation}

 \begin{thm}
  If $\la_2-\la_1>1$ then the equation (\ref{E1.1}) has only $0$ solution, and the solution is non-degenerate and the 'L-S Degree' is $-1$.
 \end{thm}
 
 \begin{proof}
  First note that any non-zero $H_0^1$ solution of the equation (\ref{E1.1}) has to change sign. Multiplying (\ref{E1.1}) by $\phi_1$ and 
  integrating by parts we get \be \int_B g(u)\phi_1=0.\notag\ee hence $u$ has to change sign. Let $u\neq 0$ be a solution of (\ref{E1.1}). 
  Defining $g(u)/u=1$ at $u=0$ we can re-write the equation (\ref{E1.1}) as
  \be -\De u=\Big[\la_1+\f{g(u)}{u}\Big]u.\label{E1.2}\ee Note that $g(u)/u\le 1$ for all $x\in B_1$ and $<1$ on a positively measured 
  subset of $B_1$,(as $u$ is non-zero solution). Consider the following eigenvalue problems 
  
  \begin{equation}
  \begin{array}{lll}
  -\De u = \mu\Big[\la_1+\f{g(u)}{u}\Big]u,  \text{  in } B,   &u=0 \text{ on } \pa B\label{E1.3}\\
  \end{array}
  \end{equation}
  
  \begin{equation}
  \begin{array}{lll}
  -\De u = \mu\la_2u,  \text{ in } B,  & u=0 \text{ in } \pa B\label{E1.4}\\
  \end{array}
  \end{equation}
  
  Note that $\la_1+\f{g(u)}{u}\le\la_2$ and the strict inequality holds on a +vely measured set. Hence we get 
  \be \mu_k\Big[\la_1+\f{g(u)}{u}\Big]>\mu_k(\la_2), \forall k.\label{2.2}\ee Now $u$ being a sign changing solution and 
  $\mu_2(\la_2)=1$ implies that $\mu_k\Big[\la_1+\f{g(u)}{u}\Big]>1, \forall k\ge 2$. Now $u$ being a sign-changing solution of 
  (\ref{E1.3}) we have $\mu_k\Big[\la_1+\f{g(u)}{u}\Big]=1$, which is contradictory. Hence $0$ is only solution of (\ref{E1.1}).
  
  The linearized equation of (\ref{E1.3}) at $0$ is \be -\De v = (\la_1+1)v,  \text{ in } B, v=0 \text{ in } \pa B.\label{E1.5}\ee Now 
  $\la_2>\la_1+1$ implies $0$ is the only solution of (\ref{E1.5}) and hence $0$ is the non-degenerate solution of (\ref{E1.1}).
  Also note that $0$ being only solution of the equation, the degree of the solution is $-1$.
  
 \end{proof}
 \begin{remark}
  In the context of our theorem, there is no loss of generality by assuming $\la_2-\la_1>1$, since we can always replace 
  $g(u)$ by $\eps g(u)$ for the comparison equation.
 \end{remark}

 As we have mentioned earlier that we shall use homotopy argument to prove our result and hence we consider the following equation 
 \begin{equation}
  \left\{\begin{aligned}
  -\De u_t &= \la_1 u_t+t(e^{u_t}+f)+(1-t)g(u_t),  &&\mbox{ \qquad in } B \label{E1.6}\\
   u_t &= 0  &&\mbox{   \qquad on } \partial B\\
  \end{aligned}\right.
 \end{equation}
   
 We shall show that the solutions $u_t$ is bounded uniformly in $L^\I(B)$. 
 
 \begin{lemma}\label{L2.1}
  Let $R_t$ be a critical point of $u_t$ such that $\exists \ t_0\in[0,1]$ with $\lim_{t\to t_0}u_t(R_t)\to\I$. Then $R_t\to 0$ as $t\to t_0$.
 \end{lemma}
 
 \begin{proof}
  $u_t$ is radial. Hence from the equation (\ref{E1.6}) we have 
  \be -(ru_t')'=\la_1 r u_t + t(e^{u_t}+f)r+ (1-t)g(u_t)r.\notag\ee The first eigenfunction $\phi_1$ of $\De$ is also radial. So multiplying 
  the above by $\phi_1$ and integrating it by parts over $[R_t,1]$ we get
  \be \int_{R_t}^1\la_1 r u_t\phi_1dr+R_tu_t(R_t)\abs{\phi_1'(R_t)}=\int_{R_t}^1\la_1 r u_t\phi_1dr+\int_{R_t}^1[t(e^{u_t}+f)+ (1-t)g(u_t)]\phi_1rdr\notag\ee
   Hence taking $A:=\{x: R_t<\abs{x}<1\}$ we have
  \begin{align}
   R_t\abs{u_t(R_t)}\abs{\phi_1'(R_t)}
   &\le\abs{\int_{R_t}^1t(e^{u_t}+f)\phi_1rdr}+\int_{R_t}^1 (1-t)\abs{g(u_t)}\phi_1rdr\notag\\
   &\le\f{1}{2\pi}\abs{\int_A t(e^{u_t}+f)\phi_1dx}+C\notag\\
   &\le\f{1}{2\pi}\abs{\int_B t(e^{u_t}+f)\phi_1dx}+\f{1}{2\pi}\int_{B\backslash A} t\abs{f}\phi_1dx+C\notag\\
   &\le\f{1}{2\pi}\int_B (1-t)\abs{g(u_t)}\phi_1dx+C\notag\\
   &\le C\notag
  \end{align}
  Now if $R_t \nrightarrow 0$ as $t\to t_0$, we have $\abs{\phi_1'(R_t)}>0, \forall t\in (t_0-\eps, t_0+\eps)\cap I$. So we have
  \be \abs{u_t(R_t)}\le \f{C}{R_t\abs{\phi_1'(R_t)}}\le C_1, \forall t\in (t_0-\eps, t_0+\eps)\cap I.\notag\ee Which is contradictory. And hence
  we have the result.
  \end{proof}

 Let us write $u_t$ as \be u_t=T_t\phi_1+\om_t,\label{2.3}\ee such that $\int_B \om_t\phi_1=0$. Then $\om_t$ satisfies
  \begin{equation}
  \left\{\begin{aligned}
  -\De \om_n &= \la_1\om_n+t_n(e^{u_n}+f)+(1-t_n)g(u_{t_n}),  &&\mbox{ \qquad in } B \label{E1.7}\\
  \om_n &= 0  &&\mbox{   \qquad on } \partial B\\
  \end{aligned}\right.
 \end{equation}
 \begin{remark}
  The same argument of Lemma \ref{L2.1} is also valid for a critical point $R_t$ of $\om_t$, i.e. if $\om_t(R_t)\to\I$ then $R_t\to 0$
 \end{remark}

 \begin{lemma}\label{L2.2}
  $\om_t$ is uniformly $L^2(B)$ bounded. i.e. \be \sup_{t\in[0,1]}\norm{\om_t}_{L^2(B)}<\I\label{2.4}\ee
 \end{lemma}
 \begin{proof}
  If possible let us assume there is a sequence $t_n$ such that $\norm{\om_{t_n}}_{L^2(B)}\to\I$. Let us denote $\om_n=\om_{t_n}$.
  Let $\tilde{\om}_n$ satisfies
  \begin{equation}
  \left\{\begin{aligned}
  -\De \tilde{\om}_n &= \la_1\om_n+t_nf+(1-t_n)g(u_{t_n}),  &&\mbox{ \qquad in } B \label{E1.8}\\
   \tilde{\om}_n &= 0  &&\mbox{   \qquad on } \partial B\\
  \end{aligned}\right.
 \end{equation}
  Divide $\ref{E1.8}$ by $\norm{\om_n}_2$ and we have
\begin{equation}
   \left\{\begin{aligned}
  -\De \big(\tf{\tilde{\om}_n}{\norm{\om_n}_2}\big) &= \la_1\tf{\om_n}{\norm{\om_n}_2}+t_n\tf{f}{\norm{\om_n}_2}+(1-t_n)\tf{g(u_{t_n})}{\norm{\om_n}_2},  &&\mbox{ \qquad in } B \label{E1.9}\\
   \tf{\tilde{\om}_n}{\norm{\om_n}_2} &= 0  &&\mbox{   \qquad on } \partial B\\
  \end{aligned}\right.
 \end{equation}
 Note that $\norm{\f{\om_n}{\norm{\om_n}_2}}=1$ and 
 \be\lim_{n\to\I}t_n\norm{\f{f_n}{\norm{\om_n}_2}}_2=\lim_{n\to\I}(1-t_n)\norm{\f{g(u_{t_n})}{\norm{\om_n}_2}}_2=0\notag\ee
 Hence by regularity we have $\f{\tilde{\om}_n}{\norm{\om_n}_2}\in H^2\cap H_0^1$ $\norm{\f{\tilde{\om}_n}{\norm{\om_n}_2}}_{H^2}<C,\forall n$, and hence
 we have $\f{\tilde{\om}_n}{\norm{\om_n}_2}\rightharpoonup L$ in $H_0^1(B)$. Also note that $\norm{\f{\om_n}{\norm{\om_n}_2}}_2=1$ implies there
 is $N\in L^2(B)$ such that $\f{\om_n}{\norm{\om_n}_2}\rightharpoonup N$ in $L^2(B)$ and we have 
 
 \begin{equation}
  \begin{array}{lll}
  -\De L = \la_1N,  &\mbox{ \qquad in } B \label{E1.10}\\
   L = 0  &\mbox{   \qquad on } \partial B\\
  \end{array}
 \end{equation}
 Multiplying (\ref{E1.10}) by $\phi_1$ and integrating by parts we get \be \int_B(L-N)\phi_1=0.\label{2.5}\ee Now note that 
 \begin{equation}
  \begin{array}{lll}
  -\De (\om_n-\tilde{\om}_n) = t_ne^{u_{t_n}}>0,  &\mbox{ \qquad in } B \label{E1.11}\\
   \om_n-\tilde{\om}_n = 0  &\mbox{   \qquad on } \partial B\\
  \end{array}
 \end{equation}
 Hence by maximum principle we have $\om_n-\tilde{\om}_n\ge 0$ and hence 
 \begin{align}
  &\int_B\Big(\f{\om_n}{\norm{\om_n}_2}-\f{\tilde{\om}_n}{\norm{\om_n}_2}\Big)\phi_1\ge 0\notag\\
  &\lim_{n\to \I}\int_B\Big(\f{\om_n}{\norm{\om_n}_2}-\f{\tilde{\om}_n}{\norm{\om_n}_2}\Big)\phi_1\ge 0\notag\\
  &\int_B(N-L)\phi_1\ge 0\label{2.6}
 \end{align}
 Hence from (\ref{2.5}) and (\ref{2.6}) we have $N=L$. So $L$ satisfies 
 \begin{equation}
  \begin{array}{lll}
  -\De L = \la_1L,  &\mbox{ \qquad in } B \label{E1.12}\\
   L = 0  &\mbox{   \qquad on } \partial B\\
  \end{array}
 \end{equation}  
 And we get $L=l\phi_1$ for some $l\in \R$ and we have $\f{\om_n}{\norm{\om_n}_2}\rightharpoonup l\phi_1$ and hence $l=0$, also we have
 \be \f{\om_n}{\norm{\om_n}_2} \rightharpoonup 0 \text{ in } L^2(B)\text{, and } \f{\tilde{\om}_n}{\norm{\om_n}_2} \rightharpoonup 0 \text{ in } H^2(B)\notag\ee
 Then by compact embedding we get $\f{\tilde{\om}_n}{\norm{\om_n}_2}\to 0$ in $H_0^1(B)$ and hence $\f{\tilde{\om}_n}{\norm{\om_n}_2}\to 0$ in 
 $C^0(\bar{B})\cap C^1(B)$. 
 
 Note that $\phi_1$, the first positive eigenfunction 
 lies in the interior of the cone of positive functions in the space $C^0(\bar{B})\cap C^1(B)$.
 So we have \be \f{\tilde{\om}_n}{\norm{\om_n}_2}\le C\phi_1\label{2.8}\ee for some positive constant $C_1$. Now $\tilde{\om}_n$ satisfies
 \be -\De \tilde{\om}_n=\la_1 \om_n+t_n f +(1-t_n)g(u_n).\notag\ee Multiplying both sides of the above by $\om_n$ and integrating by parts we get
 \begin{align}
  &\int_B \tilde{\om}_n\Big(\la_1 \om_n+t_n(e^{u_n} + f)+(1-t_n)g(u_n)\Big)dx\notag\\
  &= \la_1\norm{\om_n}^2+\int_B \Big(t_n f\om_n+(1-t_n)g(u_n))\om_n\Big)dx\notag\\
  i.e. & \int_B \la_1\f{\tilde{\om}_n}{\norm{\om_n}_2}\f{\om_n}{\norm{\om_n}_2}+\f{1}{\norm{\om_n}_2^2}\int_B \Big(t_n\tilde{\om}_ne^{u_n} + t_n\tilde{\om}_nf+(1-t_n)\tilde{\om}_ng(u_n)\Big)dx\notag\\
  &=\la_1+\int_B t_n \f{f}{\norm{\om_n}_2}\f{\om_n}{\norm{\om_n}_2}dx+(1-t_n)\int_B\f{g(u_n)}{\norm{\om_n}_2}\f{\om_n}{\norm{\om_n}_2}dx\label{2.9}
 \end{align}
 The LHS. of (\ref{2.9}) can be represented as
 \be \int_B \la_1\f{\tilde{\om}_n}{\norm{\om_n}_2}\f{\om_n}{\norm{\om_n}_2}+\f{t_n}{\norm{\om_n}_2}\int_B e^{u_n}\Big(\f{\tilde{\om}_n}{\norm{\om_n}_2}-C_1\phi_1\Big) + t_n\int_B\f{\tilde{\om}_n}{\norm{\om_n}_2}\f{f}{\norm{\om_n}_2}\notag\ee
 \be +(1-t_n)\int_B \f{g(u_n)}{\norm{\om_n}_2}\f{\tilde{\om}_n}{\norm{\om_n}_2}dx + C_1t_n\int_B e^{u_n}\f{\phi_1}{\norm{\om_n}_2}dx\notag\ee
 Now multiplying (\ref{E1.6}) by $\phi_1$ and integrating by parts we get 
 \be -t_n\int_B e^{u_n}\phi_1=t_n\int_B f\phi_1dx+(1-t_n)\int g(u_n)\phi_1dx.\label{2.10}\ee So we have $\abs{t_n\int_B e^{u_n}\phi_1}<\I$. Hence
 \be \lim_{n\to\I}t_n\int_B \f{e^{u_n}\phi_1}{\norm{\om_n}_2}\to 0.\label{2.11}\ee Then from (\ref{2.8}) we have as $n\to\I$,
 LHS. $\le 0$. Similarly we can show that as $n\to\I$ RHS. $\to\la_1>0$ which is contradictory. Hence we have $\norm{\om_t}_2$ bounded
 uniformly. 
 \end{proof}
 
 \begin{lemma}\label{LTbound}
  $T_t$ is bounded.
 \end{lemma}
 \begin{proof}
  We have taken $u_n=T_n\phi_1+\om_n$ and $\norm{\om_n}_2<\I$. 
 Define $\tilde{\om}$ as in Lemma \ref{L2.2}. We have seen that $\widetilde{\om}_n\in H^2(B)\cap H_0^1(B)$.
  And hence by Sobolev embedding theorem $\widetilde{\om}_n\in C^1(B)\cap C^0(\bar{B})$.  Now
 \begin{equation}
  \begin{array}{lll}
  -\De (\om_n-\widetilde{\om}_n) = t_ne^{u_n},  &\mbox{ \qquad in } B \label{E1.14}\\
   \om_n-\widetilde{\om}_n = 0  &\mbox{   \qquad on } \partial B\\
  \end{array}
 \end{equation}  
 By maximum principle we have $\om_n-\widetilde{\om}_n > 0$ in $B$. So $\om_n$ is bounded from below uniformly on $n$.
 
 If possible let us suppose that there is $t_n$ such that $T_n:=T_{t_n}\to\I$. Let us first show that $t_n\to 0$ as $n\to\I$. If not, let 
 up to a subsequence $t_n\to t_0\neq 0$. So for large $n$ we have from (\ref{2.10})
 \be \int_B\phi_1(e^{u_n}+f)+\f{1-t_n}{t_n}\int_Bg(u_n)\phi_1=0.\notag\ee Note that in any compact set $K\subset B$, $\phi_1e^{u_n}\to \I$
 uniformly as $\om_n$ bounded below. So the above inequality can't hold as all other terms are bounded. Hence $\lim_{n\to\I}t_n\to 0$. 
 
 Now let us show that for n large $u_n\ge0$. Divide $[0,1]$ into two fixed intervals $[0,1-\delta]$ and $(1-\delta,1]$, for some small positive
 number $\delta$. Note that there is $N_1$ such that $u_n\ge 0$ in $[0,1-\delta]$ for all $n\ge N_1$. Form (\ref{E1.14}) and using Hopf maximum 
 principle we have \be \f{\pa \widetilde{\om}_n}{\pa\eta}\ge \f{\pa \om_n}{\pa\eta}.\notag\ee Using elliptic regularity and Sobolev embedding 
 we have $\f{\pa \widetilde{\om}_n}{\pa\eta}$ is bounded uniformly on $\pa A$. And hence $\f{\pa {\om}_n}{\pa\eta}$ is bounded  uniformly on 
 $\pa A$. Note that $\phi_1'(1)<0$, implies there is $N$ large such that \be u_n'(1)=T_n\phi_1'(1)+\om_n'(1)<0.\label{2.12}\ee Hence $u_n$ is 
 positive near the boundary for $n\ge N$. Hence $u_n$ is positive near the boundary for $n$ large. Now let $u_n$ changes sign. Define
 \be a_n=\sup\{r\in(0,1):u_n(a_n)=0\}.\notag\ee Clearly $u_n'(a_n)\ge 0$. First note that $\lim_{n\to\I} a_n\to 1$, if not let up to a 
 subsequence $a_n\to a<1$. Now $u_n(a_n)=T_n\phi_1(a_n)+\om_n(a_n)\ge T_n\phi_1(a+1/n)+\om_n(a_n)\to\I$ as $n\to\I$, which is a 
 contradiction. Now note that \be u_n'(r)=T_n\phi_n'(r)+\om_n'(r)\le T_n\phi_n'(1-\delta)+\om_n'(r)\to -\I, \text{ as } n\to\I,\notag\ee
 for all $r\in(1-\delta,1)$. Now $a_n\to 1$ and $u_n'(a_n)\ge 0$ contradicts the above. Hence $u_n\ge 0$ for $n$ large.
 
 Now $u_n\ge 0$ implies $g(u_n)\ge 0$ also $e^{u_n}\to\I$ on a positive measured set, and we have 
 \be t_n\int_B e^{u_n}\phi_1+t_n\int_B f\phi_1dx+(1-t_n)\int g(u_n)dx>0,\notag\ee contradicting (\ref{2.10}). Hence $T_n\nrightarrow \I$.
 
 If possible let $T_n\to -\I$. Let us write $u_n=-T_n\phi_1+\om_n$. Then note that $T_n\to \I$. First note that 
 \be \int_A \om_n^+\phi_1=\int_A \om_n^-\phi_1<\I.\notag\ee Hence $\lim_{n\to\I}\mu\{x:\om_n^+(x)>n\}=0$. Then we have as 
 $T_n\to \I$, $\lim_{n\to\I} \mu\{x:-T_n\phi_1+\om_n(x)>-\pi\}=0$. And thus $\lim_{n\to\I}\int_A g(u_n)\phi_1=0$, as $g$ and $\phi_1$
 are both bounded.Then from (\ref{2.10}) we have either $\int_B e^{u_n}\phi_1\to-\int_Bf\phi_1$ or $t_n\to 0$ and in this case 
 $\lim_{n\to \I}\int_A t_ne^{u_n}\phi_1=0$.
 
 Case I : Let us first assume $t_n\to 0$. Define $v_n:=\om_n-\widetilde{\om}_n$. Then from (\ref{E1.14}) we have
  \begin{equation}
  \begin{array}{lll}
  -\De v_n = V_ne^{v_n},  &\mbox{ \qquad in } B \label{E1.15}\\
   v_n = 0  &\mbox{   \qquad on } \partial B\\
  \end{array}
 \end{equation} 
 Where $V_n=t_ne^{-T_n\phi_1+\widetilde{\om}_n}$. Note that $\norm{\widetilde{\om}_n}_\I<\I$ uniformly on $n$. And hence we have 
 $\norm{V_n}_\I<\I$, uniformly on $n$. Now $V_ne^{v_n}=t_ne^{u_n}$.
 
 Let us first assume that, there is $0<a<1$ such that $\om_n<0$ on $(a,1)$ for all $n$. Then from (\ref{2.10}), we have 
 \be \int_{B(0,a)}t_ne^{u_n}\le \f{1}{\phi_1(a)}[-t_n\int_B f\phi_1+(1-t_n)\int_Bg(u_n)\phi_1].\notag\ee So 
 $\lim_{n\to\I}\int_{B(0,a)}t_ne^{u_n}=0$. Now as in $(a,1), \om_n<0$, by choosing $a$ properly and $n$ large we have 
 \be \int_B V_ne^{v_n}=\int_B t_ne^{u_n}<4\pi.\label{2.13}\ee So by the result of Brezis Merle \cite{MR1132783} we have $\norm{\om_n-\widetilde{\om}_n}_\I<C$, for 
 some positive constant $C$, for all $n$. And hence we have $\norm{\om_n}<C$, for all $n$. Then using regularity form (\ref{E1.7}) we get
 $\norm{\om_n}_{C^1(\bar{B})}<C$, for all $n$. Now $\phi_1$ being in the interior of the cone of positive functions in ${C^1(\bar{B})}$
 we have \be u_n=-T_n(\phi_1-\f{\om_n}{T_n})<0, \text{ for $n$ large.}\notag\ee 
 Now choose an interval $(a,1)$ such that 
 \be \int_a^1\phi_1 dr +1/4\int_0^1 f\phi_1 dr<0.\label{2.14}\ee also in the interval $[0,a]$, $u_n\downarrow 0$ uniformly. Hence for $n$ large,
 \be \int_0^a e^{u_n}\phi_1 dr +1/4\int_0^1 f\phi_1 dr<0.\label{2.15}\ee Combining (\ref{2.14}) and (\ref{2.15}) we get
 \be \int_0^1 e^{u_n}\phi_1 dr +\int_0^1 f\phi_1 dr<0\notag\ee as $\int_0^1 f\phi_1<0$. Now note that $g(u_n)\le 0$ for $n$ large, hence
 $\int_0^1g(u_n)\le0$. So we have \be t_n\int_0^1 (e^{u_n}+f)\phi_1dr+(1-t_n)\int_0^1 g(u_n)dr<0.\notag\ee contradicting (\ref{2.10}). So $\om_n$
 has to positive value in $(a,1)$ for $n$ large.
 
 Now we shall show that in $(a,1)$, $\norm{\om_n}_{L^\I(a,1)}<C$ for some positive constant $C$ for all $n$ large. If not we have for any
 $M_n\to\I$ there is $\delta_n>0$ such that $\mu\{x : \om_n(x)>M_n\}\ge\delta_n$ and there is $p_n\in (a,1)$ such that $\om_n(p_n)>M_n$. Let
 $p_n\to p$ as $n\to\I$(up to a subsequence). Take $0<p'<\inf\{p,a\}$. claim that in $(p',a)$, $\om_n\ge M_n$ for all $n$. If not, we shall find
 $q_n$ such that $\om_n(q_n)<M_n,\ \forall n$. Hence there is a point of maxima $R_n$ of $\om_n$ in $(p',a)$ with 
 $lim_{n\to\I}\om_n(R_n)\to\I$ and $R_n\nrightarrow 0$. Now similarly as proved in lemma[\ref{L2.1}] we can show the same result for $\om_n$.
 Which gives us a contradiction. 
 
 Now choose $M$ large enough so that $\phi(a)M(p-p')>\int_{B_1(0)}\om_n^-\phi_1$. Then we have for $n$ large $M_n>M$ and
 \be \int_{B_a(0)}\om_n^+\phi_1\ge \phi(a)\int_{p'}^aMdx=\phi(a)M(p-p')>\int_{B_1(0)}\om_n^-\phi_1.\notag\ee Which is contradictory. So in $(a,1]$, 
 $\norm{\om_n}_{L^\I(a,1)}<C$.
 
 Now from \ref{2.10} we have \be \int_{B_a(0)}t_ne^{u_n}\le\f{1}{\phi_1(a)}\Big[t_n\int_B \abs{f}\phi_1dx+(1-t_n)\int_B \phi_1 dx.\Big]<\I\notag\ee
 and \be \int_{B\backslash B_a(0)}t_ne^{u_n}=\int_{B\backslash B_a(0)}t_ne^{-T_n\phi_1}e^C<C.\notag\ee
 So using theorem[\ref{Th1.1}] we conclude that $\om_n-\widetilde{\om}_n\in L^\I(B)$, and $\norm{\om_n-\widetilde{\om}_n}L^\I(B)\le C$, and
 thus we have $\norm{\om_n}L^\I(B)<C$. Using regularity we have for $n$ large $\om_n\in C^1_0(\bar{B})$. Now using the same cone condition 
 we have $-T_n\phi_1+\om_m<0$ and $-T_n\phi_1+\om_m\to-\I$ in any compact subset of $B$. Hence $(1-t_n)\int_B g(u_n)\phi_1\le 0$ for $n$ large 
 and $\int_K e^{u_n}\to 0$ as $n\to \I$ for any compact $K\subset B$. Using the fact $\int_B f\phi_1<0$ we have 
 \be t_n\int_B (e^{u_n}+f)\phi_1dx+(1-t_n)\int_B g(u_n)\phi_1dx<0.\notag\ee Which is contradictory. Hence $T_n\nrightarrow \I$. 
 
 Case II : Now let $\int_B e^{u_n}\phi_1\to-\int_Bf\phi_1$. We have $-\int_B f\phi_1<4\pi$. So for $n$ large $t_n\int_B e^{u_n}\phi_1<4\pi$.
 Now for any $0<a<1$ we have shown that $\norm{\om_n}_{L^\I(a,1)}<C$. As $T_n\to-\I$, we have $\norm{\f{\om_n}{T_n}}_{L^\I(a,1)}\to 0$
 as $n\to\I$. Hence $\int_{B\backslash B_a(0)}t_ne^{u_n}\to 0$ as $n\to\I$. Now choose $\delta$ such that 
 $-\int_b f\phi_1<4\pi(1-\delta)$. Choose $a$ such that $\phi(a)=1-\delta$. Then 
 \be t_n\int_{B_a(0)}e^{u_n}\le\f{t_n}{\phi_1(a)}\int_{B_a(0)}e^{u_n}\phi_1 dx<4\pi.\notag\ee Combining both the integrals we have 
 $t_n\int_B e^{u_n}<4\pi$ for $n$ large. And similarly as above we can arrive at the same contradiction. 
 
 Hence in both  the cases we have $T_n<C$ for some positive constant $C$.
 
 \end{proof}

 \begin{thm}\label{Th2.2}
  $\norm{u_t}_{L^\I(B)}<C$, for some positive constant $C$.
 \end{thm}
 \begin{proof}
  If not then there exists a sequence $t_n$ and a sequence of solutions $u_n$ such that $\norm{u_n}_{L^\I(B)}\to\I$ as $n\to\I$.
  Now expressing $u_n=t_n\phi_1+\om_n$, we have shown that $T_n$ is bounded. Hence $\norm{\om_n}_{L^\I(B)}\to\I$ as $n\to\I$. Also 
  we know that $R_n\to 0$ where $\om_n(R_n)\to\I$ and $\om_n'(R_n)=0$. And as before we can show that for any sequence $q_n$ with 
  $\om_n(q_n)\to\I$, $q_n\to 0$. Hence for any $1>\delta>0$ 
   \be\norm{u_n}_{L^\I(B\backslash B_\delta(0))}<C_\delta,\label{2.16}\ee for some $C_\delta>0$ \par
  
  Now we shall establish a contradiction for two different cases.
  
  Case I : Let $t_n\to t_0\neq 0$. Let $v_n=\om_n-\widetilde{\om}_n$. Then $v_n$ satisfies \ref{E1.15}. From \ref{2.10} we have 
  \be lim_{n\to\I} \int_B e^{u_n}\phi_1=-\int_B f\phi_1+\f{1-t_0}{t_0}\lim_{n\to\I}\int_B g(u_n)\phi_1.\label{2.17}\ee Hence there is a constant 
  $C$ such that $\int_B e^{u_n}\phi_1<C$ uniformly for large $n$. Also for any $0<a<1$ we have shown that $\om_n$ is uniformly bounded 
  in $B\backslash B_a(0)$. Hence  \be \int_{B\backslash B_a(0)} e^{v_n}=\int_{B\backslash B_a(0)} e^{\om_n-\widetilde{\om}_n}<C,\notag\ee 
  as $\widetilde{\om}_n$ bounded uniformly in $B$. Now
  \be \int_{B_a(0)}e^{v_n}=\int_{B_a(0)} e^{u_n}e^{-(T_n\phi_1+\widetilde{\om}_n)}\le\f{C}{\phi_1(a)}\int_B e^{u_n}\phi_1\le C.\label{2.18}\ee
  Combining (\ref{2.17}),(\ref{2.18}) we have $\norm{v_n}_{L^1(B)}<C$ uniformly in $n$. Also note that $V_n\ge 0$ and $\norm{V_n}_{L^\I(B)}<C$
  for all $n$. So using theorem.3(sec III.2) of \cite{MR1132783} we get $v_n$ (up-to a subsequence) is bounded in $L^\I_{loc}(B)$. Hence
  $\norm {\om_n}_{L^\I(B_a(0))}<C$, which gives along with (\ref{2.16}), $\norm {\om_n}_{L^\I(B)}<C$, which is contradictory to our 
  assumption.
  
  Case II : Let $t_n\to 0$. Then fron \ref{2.10} we have
  \be \lim_{n\to\I}t_n\int_{B}e^{u_n}\phi_1=-\lim_{n\to\I}\int_{B}g(u_n)\phi_1.\label{2.19}\ee Now from (\ref{2.16}) we get for 
  any $0<a<1$, \be\lim_{n\to\I}t_n\int_{B\backslash B_a(0)}e^{u_n}=0.\label{2.20}\ee Now choose $a$ such that $\phi_1(a)=1/2$. 
  Then from (\ref{2.19}) we have 
  \be t_n\int_{B_a(0)}e^{u_n}<\f{t_n}{\phi_1(a)}\int_B e^{u_n}\phi_1<2\int_B\phi_1+o(\f{1}{n})<3\pi.\label{2.21}\ee 
  Combining (\ref{2.20}), (\ref{2.21}) we have \be \int_B V_ne^{v_n}=t_n\int_B e^{u_n}<4\pi.\notag\ee Using Cor.3(Sec III.1) of \cite{MR1132783} we 
  get $\norm{v_n}_{L^\I(B)}<C$. That is $\norm{\om_n}_{L^\I(B)}<C$, for all $n$, which is contradictory to our assumption.  
 \end{proof}
 
 \textbf{Proof of theorem(\ref{Th1.1}):}
 \begin{proof}
  In theorem(\ref{Th2.2}) we have shown $\norm{u_t}_{L^\I}$ is uniformly bounded. Using regularity we get $\norm{u_t}_{C^{1,\alpha}}<C$,
  for some positive constant $C$. Now take $\Om\subset C^{1,\alpha}_{rad}$, where 
  $\Om:=\{u : u(x)=u(\abs{x}), u\in C^{1,\alpha}(B)\cap C^0(B),\norm{u}_{C^{1,\alpha}(B)}<C\}$. Now take
  \be S_t=I-\Delta^{-1}\{\la_1 I+t(exp\circ I+f)+(1-t)g\circ I\}.\notag\ee Note that $0\notin S(\pa\Om)$ for all $t$. So using homotopy 
  invariance we get $deg(\Om, S_0, 0)=deg(\Om, S_1, 0)=-1$. Hence the equation(\ref{E0.1}) has a radial solution. Also from the equation it is 
  obvious that the solution is nontrivial for $f\neq -1$. 
 \end{proof}
 
 \section{extension to $8\pi$}
 
  In theorem (\ref{Th1.1}) we have seen that the result is valid for $0<-\int_B f\phi_1<4\pi$. In this section we shall extend the result for
  $0<-\int_B f\phi_1<8\pi$. We recall the result of Li and Shafrir \cite{MR1322618},
  \begin{theorem}\label{ThLS}
   Suppose $V_n\in C^0(\bar{\Om})$, $V_n\ge 0$ in $\Om$ and $V_v\to V$ in $C^0(\bar{\Om})$. Let $\{u_n\}$ be a sequence of solution of (\ref{EBM})
   with $\norm{e^{u_n}}_{L^1(\Om)}\le C_1$ for some positive constant $C_1$. Assume alternative $(c)$ of theorem (\ref{ThBM.3}) holds. Then for each
   $i$, $\al_i=8\pi m_i$ for some positive integer $m_i$.
  \end{theorem}
  In lemma \ref{L2.1} we have seen that the blow-up can occur only at the origin. Let us consider a small ball $B_\de$ ($\de\ll1$) around the origin.
  Consider the equation \be -\De v_n = V_ne^{v_n} \text{ in }B_\de,\label{E3.1}\ee as in (\ref{E1.15}). Also in the proof of lemma \ref{LTbound}
  we have seen that $T_t$ can not go to $\I$ irrespective of the value of the value $-\int_B f\phi_1$. We have used the integral value 
  $-\int_B f\phi_1$ to prove $T_t$ is also bounded below. In that case when $t_n\to 0$ we have used the result of Brezis-Marle (\ref{ThBM.2}). 
  It is the other case when $\int_B e^{u_n}\phi_1\to-\int_Bf\phi_1$ we have used the integral value to get the lower bound for $T_t$. Then
  provided the condition $-\int_Bf\phi_1<8\pi$ we can show as lemma \ref{LTbound} that \be \int_B V_ne^{v_n}=\int_B t_ne^{u_n}<8\pi.\label{3.2}\ee
  
  We have seen that the possible blowup set $S$ contains only the origin i.e. $S=\{0\}$. If $v_n$ is blowing up at the origin then by the
  result of Li and Shafrir (Theorem \ref{ThLS}) we have up-to a subsequence $V_ne^{v_n}\to \al_0\de_{a_0}$ with $\al_0=8\pi m_0$, in the sense 
  of measure, for some positive integer $m_0$. which implies $\int_B V_ne^{v_n}\ge 8\pi m_0$ which is contradictory. So we have 
  $\norm{v_n}_{L^\I}<C$. The rest of the proof follows similarly. 
  
  \begin{remark}
   The condition $0<-\int_B f\phi_1<8\pi$ seems to be optimal even if we have not been able to establish this. In \cite{MR2287880} the authors 
   have analyzed the same type of equation (see equation (1.6) of \cite{MR2287880})
   as equation (\ref{E1.15}) in our case leading to blow up.
  \end{remark}

\bibliographystyle{plain}


\begin{thebibliography}{10}

\bibitem{MR1132783}
Ha{\"{\i}}m Brezis and Frank Merle.
\newblock Uniform estimates and blow-up behavior for solutions of {$-\Delta
  u=V(x)e^u$} in two dimensions.
\newblock {\em Comm. Partial Differential Equations}, 16(8-9):1223--1253, 1991.


\bibitem{MR1989831}
Mabel Cuesta, Djairo~G. de~Figueiredo, and P.~N. Srikanth.
\newblock On a resonant-superlinear elliptic problem.
\newblock {\em Calc. Var. Partial Differential Equations}, 17(3):221--233,
  2003.

\bibitem{MR2287880}
Manuel del Pino and Claudio Mu{\~n}oz.
\newblock The two-dimensional {L}azer-{M}c{K}enna conjecture for an exponential
  nonlinearity.
\newblock {\em J. Differential Equations}, 231(1):108--134, 2006.

\bibitem{MR1322618}
Yan~Yan Li and Itai Shafrir.
\newblock Blow-up analysis for solutions of {$-\Delta u=Ve^u$} in dimension
  two.
\newblock {\em Indiana Univ. Math. J.}, 43(4):1255--1270, 1994.


\bibitem{MR491053}
P.~J. McKenna and J.~Rauch.
\newblock Strongly nonlinear perturbations of nonnegative boundary value
  problems with kernel.
\newblock {\em J. Differential Equations}, 28(2):253--265, 1978.


\end{thebibliography}

\end{document}